\documentclass[11pt,a4paper]{article}

\usepackage{inputenc}
\usepackage{amsmath}
\usepackage{bm}
\usepackage{bbold}
\usepackage{amsthm}

\usepackage{hyperref}

\title{Efficient Algorithms\\ for Tandem Queueing System Simulation\thanks{Applied Mathematics Letters, 1994. Vol.~7, no.~6, pp.~45-49}
}

\author{Sergei M. Ermakov and Nikolai K. Krivulin\thanks{Faculty of Mathematics and Mechanics, St.~Petersburg State University, Bibliotechnaya Sq.2, Petrodvorets, St.~Petersburg, 198904 Russia}}

\date{}

\newtheorem{theorem}{Theorem}
\newtheorem{lemma}[theorem]{Lemma}
\newtheorem{corollary}[theorem]{Corollary}

\begin{document}

\maketitle

\begin{abstract}
Serial and parallel algorithms for simulation of tandem queueing systems with
infinite buffers are presented, and their performance is examined. It is
shown that the algorithms which are based on a simple computational
procedure involve low time and memory requirements.
\\

\end{abstract}

\section{Introduction}
The simulation of a queueing system is normally an iterative process which
involves generation of random variables associated with current events in the
system, and evaluation of the system state variables when new events occur
\cite{1,2,3,4,5}. In a system being simulated the random variables may represent the
interarrival and service time of customers, and determine a random routing
procedure for customers within the system with non-deterministic routing. As
state variables, the arrival and departure time of customers, and the service
initiation time can be considered.

The methods of generating random variables present one of the main issues in
computer simulation, which has been studied intensively in the literature
(see, e.g. \cite{3}). In this paper however, we assume as in \cite{2} that for
the random variables involved in simulating a queueing system, appropriate
realizations are available when required, and we, therefore, concentrate only
on algorithms of evaluating the system state variables from these
realizations.

The usual way to represent dynamics of queueing systems is based on recursive
equations describing evolution of system state variables. Furthermore, these
equations, which actually determine a global structure of changes in the state
variables consecutively, have proved to be useful in designing efficient
simulation algorithms \cite{1,2,4,5}.

In this paper we apply recursive equations to the development of algorithms
for simulation of tandem queueing systems with infinite buffers capacity. A
serial algorithm and parallel one designed for implementation on single
instruction, multiple data parallel processors are presented. These algorithms
are based on a simple computational procedure which exploits a particular
order of evaluating the system state variables from the recursive equations.
The analysis of their performance shows that the algorithms involve low time
and memory requirements.

The rest of the paper is organized as follows. In Section~2, we give
recursive equations based representation for tandem systems. These
representations are used in Section~3 to design both serial and parallel
simulation algorithms. Time and memory requirements of the algorithms are also
discussed in this section. Finally, Section~4 includes two lemmae which offer
a closer examination of the performance of the parallel algorithm.

\section{Tandem Queues with Infinite Buffers}

To set up the recursive equations that underlie the development and analysis
of simulation algorithms in the next sections, consider a series of $  M  $
single server queues with infinite buffers. Each customer that arrives into
the system is initially placed in the buffer at the $1$st server and then has
to pass through all the queues one after the other. Upon the completion of his
service at server $  i $, the customer is instantaneously transferred to
queue $  i+1 $, $  i=1,\dots,M-1 $, and occupies the $(i+1)$st server
provided that it is free. If the customer finds this server busy, he is placed
in its buffer and has to wait until the service of all his predecessors is
completed.

Denote the time between the arrivals of $j$th customer and his predecessor
by $  \tau_{0j} $, and the service time of the $j$th customer at server
$  i  $ by $  \tau_{ij} $, $  i=1,\dots,M $, $  j=1,2,\dots $.
Furthermore, let $  D_{0j}  $ be the $j$th arrival epoch to the system, and
$  D_{ij}  $ be the $j$th departure epoch from the $i$th server. We assume
that $  \tau_{ij} \geq 0 $ are given parameters, whereas $  D_{ij}  $ are
unknown state variables. Finally, we define $  D_{i0} \equiv 0  $ for all
$  i=0,\dots,M $, and $  D_{-1j} \equiv 0  $ for $  j=1,2,\ldots $.

With the condition that the tandem queueing system starts operating at time
zero, and it is free of customers at the initial time, the recursive equations
representing the system dynamics can readily be written as \cite{1,2,4}
\begin{equation}
D_{ij} = (D_{i-1j} \vee D_{ij-1}) + \tau_{ij},
\label{tag1}
\end{equation}
where $  \vee  $ denotes the maximum operator, $  i=0,1,\dots, M $,
$  j=1,2,\dots $. We can also rewrite (\ref{tag1}) in another form intended to
provide the basic representation for parallel algorithms, as
\begin{equation}
\begin{split}
B_{ij} & = D_{i-1j} \vee D_{ij-1} \\
D_{ij} & = B_{ij} + \tau_{ij},
\end{split}
\label{tag2}
\end{equation}
where $  B_{ij}  $ stands for the $j$th initiation of service at server
$  i $, $  i=0,1,\dots, M $, $  j=1,2,\dots $.

\section{Algorithms for Tandem Queueing System Simulation}

The simulation algorithms presented in this section are based on the equations
(\ref{tag1}) and (\ref{tag2}) with indices being varied in a particular order
which is illustrated in Fig.~1(a). Clearly, at each iteration $  k  $ the
variables $  D_{ij}  $ with $  i+j=k $, $ k=1,2,\ldots $, are evaluated.
They form diagonals depicted in Fig.~1(a) by arrows, for each diagonal the
direction of arrows indicates the order in which the variables should be
evaluated within their associated iterations. Note that to obtain each element
of a diagonal, only two elements from the preceding diagonal are required, as
Fig.~1(b) shows.
\begin{figure}[hhh]
\begin{gather*}
\begin{array}{ccccccccc}
  D_{01} &          & D_{02} &          & D_{03} &          & D_{04} &
         & D_{05} \\
         & \nearrow &        & \nearrow &        & \nearrow &        &
\nearrow & \\
  D_{11} &          & D_{12} &          & D_{13} &          & D_{14} &
         & D_{15} \\
         & \nearrow &        & \nearrow &        & \nearrow &        &
\nearrow & \\
  D_{21} &          & D_{22} &          & D_{23} &          & D_{24} &
         & D_{25}
\end{array}
\\
\\
\text{(a) The simulation schematic for a tandem system with $  M=2, \; N=5$.} \\
\\
\\
\begin{array}{ccccc}
          &          &          &          & D_{i-1j} \\
          &          &          & \swarrow &          \\
          &          &  B_{ij}  &          &          \\
          & \nearrow &          & \searrow &          \\
 D_{ij-1} &          &          &          & D_{ij}
\end{array}
\\
\\
\text{(b) The diagram of calculating $  D_{ij} $.} \\
\end{gather*}
\vspace{-5ex}
\caption{Simulation procedure schematics.}
\end{figure}

By applying the computational procedure outlined above, both serial and
parallel simulation algorithms which provide considerable savings in time and
memory costs may be readily designed. Specifically, the next serial algorithm
is intended for simulation of the first $  N  $ customers in a tandem
queueing system with $  M  $ servers.
\\

{\noindent \textbf{Algorithm 1.}}
\\
For each $  k=1,\ldots,M+N $, do
\\
\hphantom{For} for $  j=\max(1,k-M), \max(1,k-M)+1, \ldots, \min(k,N) $,
\\
\hphantom{For for} compute $  D_{k-jj} = D_{k-j-1j} \vee D_{k-jj-1} +
\tau_{k-jj} $.
\\

Based on Fig.~1(a) as an illustration, it is not difficult to calculate the
total number of arithmetic operations which one has to perform using
Algorithm~1. Since each variable $  D_{ij}  $ can be obtained using one
maximization and one addition, all variables with $  i=0,\ldots,M $, and
$  j=1,\ldots N $, require $  2(M+1)N  $ operations without considering
index manipulations.

Note that the order in which the variables $  D_{ij}  $ are evaluated within
each iteration is essential for reducing memory used for computations. One can
easily see that only $  O(\min(M+1,N))  $ memory locations are actually
required with this order. To illuminate the memory requirements, let us
represent Algorithm~1 in more detailed form as
\\

{\noindent \textbf{Algorithm 2.}}
\\
Set $  d_{i}=0 $, $  i=-1,0,\ldots,\min(M+1,N) $.
\\
For each $  k=1,\ldots,M+N $, do
\\
\hphantom{For} for $  j=\max(1,k-M), \max(1,k-M)+1, \ldots, \min(k,N) $,
\\
\hphantom{For for} set $  d_{k-j} \longleftarrow d_{k-j-1} \vee d_{k-j}
+ \tau_{k-j j} $.
\\

In Algorithm~2, the variable $  d_{i}  $ serves all the iterations to store
current values of $  D_{ij}  $ for all $  j=1,\ldots, N $. Upon the
completion of the algorithm, we have for server $  i  $ the $N$th departure
time saved in $  d_{i} $, $  i=0,1,\ldots, M $.

Finally, we present a parallel algorithm for tandem system simulation which is
actually a simple modification of Algorithm~1.
\\

{\noindent \textbf{Algorithm 3.}}
\\
For each $  k=1,\ldots,M+N $, do \newline
\hphantom{For} in parallel, for
               $  j=\max(1,k-M), \max(1,k-M)+1, \ldots, \min(k,N) $,
\\
\hphantom{For in parallel, for} compute
                         $  B_{k-jj} = D_{k-j-1j} \vee D_{k-jj-1} $;
\\
\hphantom{For} in parallel, for
               $  j=\max(1,k-M), \max(1,k-M)+1, \ldots, \min(k,N) $,
\\
\hphantom{For in parallel, for} compute
                         $  D_{k-jj} = B_{k-jj} + \tau_{k-jj} $.
\\

As in the case of Algorithm~1, we may conclude that Algorithm~3 entails
$  O(2\min(M+1,N))  $ memory locations. Furthermore, it is easy to
understand that Algorithm~3 requires the performance of $  2(M+N)  $
parallel operations provided $  P \geq \min(M+1,N)  $ processors are used.
Otherwise, if there are $  P < \min(M+1,N)  $ processors available, one has
to rearrange computations so as to execute each iteration in several parallel
steps. In other words, all operations within an iteration should be
sequentially separated into groups of $  P  $ operations, assigned to the
sequential steps. We will discuss time requirements and speedup of the
algorithm in this case in the next section.

\section{Performance Study of the Parallel Simulation Algorithm}

We now turn to the performance evaluation of Algorithm~3 with respect to the
number of parallel processors. Note that the results of this section are
obtained by considering the time taken to compute only the state variables
$  D_{ij} $. In other words, in our analysis we ignore the time required for
computing indices, allocating and moving data, and synchronizing processors,
which in general can have an appreciable effect on the performance of parallel
algorithms.
\begin{lemma}
To simulate the first $  N  $ customers in a tandem queue
with $  M  $ servers, Algorithm~3 using $  P  $ processors requires the
time
\begin{equation}
T_{P} = O\left(2M+2N+2\left\lfloor\frac{L_{1}-1}{P}\right\rfloor
        (L_{2}-P)\right),
\label{tag3}
\end{equation}
where $  L_{1} = \min(M+1,N) $, $  L_{2} = \max(M+1,N) $.
\end{lemma}
\begin{proof}
We start our proof with evaluating the exact number of parallel
operations to be performed when $  P  $ processors are available. As it easy
to see, at each iteration $  k $, $  k=1, \ldots, M+N $, the algorithm first
carries out in parallel a fixed number of maximizations, and then does the
same number of additions. Denote this number by $  l_{k} $. It follows from
the above description of the algorithm that the numbers $  l_{k} $,
$  k=1,\ldots, M+N $, form the sequence with elements
$$
1,2, \ldots, L_{1}-1,
\underbrace{L_{1}, \ldots, L_{1},}_{\text{$L_{2}-L_{1}+1$ times}}
	     L_{1}-1, L_{1}-2, \ldots, 1.
$$
Since $  l_{k}  $ parallel operations may be performed using $  P  $
processors in the time $  \lfloor(l_{k}-1)/P\rfloor+1 $, for the entire
algorithm we have the total time
\begin{multline}
T_{P}
= 2\sum_{k=1}^{M+N}\left(\left\lfloor\frac{l_{k}-1}{P}\right\rfloor+1\right)
\\
= 4\sum_{k=1}^{L_{1}} \left(\left\lfloor\frac{k-1}{P}\right\rfloor+1\right)
+ 2(L_{2}-L_{1}-1)\left(\left\lfloor\frac{L_{1}-1}{P}\right\rfloor+1\right).
\label{tag4}
\end{multline}

To calculate $  T_{P} $, let us first consider the sum
\begin{multline*}
\sum_{k=1}^{L_{1}} \left\lfloor\frac{k-1}{P}\right\rfloor \\
= \underbrace{1+\cdots+1}_{\text{$P$ times}}+
  \underbrace{2+\cdots+2}_{\text{$P$ times}}+ \cdots +
  \underbrace{\left\lfloor\tfrac{L_{1}-1}{P}\right\rfloor-1 + \cdots +
  \left\lfloor\tfrac{L_{1}-1}{P}\right\rfloor-1}_{\text{$P$ times}}
  \\
   +
  \underbrace{\left\lfloor\tfrac{L_{1}-1}{P}\right\rfloor + \cdots +
  \left\lfloor\tfrac{L_{1}-1}{P}\right\rfloor}_{\text
  {$L_{1}-P\lfloor(L_{1}-1)/P\rfloor$ times}} \\
= \frac{P}{2}\left\lfloor\frac{L_{1}-1}{P}\right\rfloor
  \left(\left\lfloor\frac{L_{1}-1}{P}\right\rfloor-1\right) +
  \left(L_{1}-P\left\lfloor\frac{L_{1}-1}{P}\right\rfloor\right)
  \left\lfloor\frac{L_{1}-1}{P}\right\rfloor.
\end{multline*}
Substitution of this expression into (\ref{tag4}), and trivial algebraic
manipulations give
\begin{equation}
T_{P} = 2(L_{1}+L_{2}-1) + 2\left\lfloor\frac{L_{1}-1}{P}\right\rfloor
\left(L_{1}+L_{2}-1-P-P\left\lfloor\frac{L_{1}-1}{P}\right\rfloor\right).
\label{tag5}
\end{equation}

Finally, since $  L_{1}+L_{2}-1 = M+N $, and
$  P\lfloor L/P \rfloor \sim L  $ as $  L \to \infty $, we conclude that
$$
T_{P}
= O\left(2M+2N+2\left\lfloor\frac{L_{1}-1}{P}\right\rfloor(L_{2}-P)\right)
                               \qquad \text{as $  M,N \to \infty $}. \qedhere
$$
\end{proof}

Note that in two critical cases with $  P=1  $ and $  P \geq \min(M+1,N) $,
the order produced by (\ref{tag3}) coincides with the exact times respectively
equaled $  2(M+1)N  $ and $  2(M+N) $.

\begin{lemma}
For a tandem system with $  M  $ servers, Algorithm~3
using $  P  $ processors achieves the speedup
\begin{equation}
S_{P} = O\left(\frac{M+1}{1+\lfloor M/P\rfloor}\right)
        \qquad \text{as $  N \to \infty $}.
\label{tag6}
\end{equation}
\end{lemma}
\begin{proof} To evaluate the speedup which is defined as
\begin{equation}
S_{P} = T_{1}/T_{P},
\label{tag7}
\end{equation}
first note that $  T_{1} = 2(M+1)N $.

Let us examine $  T_{P} $. Assuming $  M  $ to be fixed, and
$  N \to \infty $, we obtain
$$ L_{1} = \min(M+1,N) = M+1, \qquad L_{2} = \max(M+1,N) = N. $$
In that case, from (\ref{tag5}) we have
\begin{align*}
T_{P} & = 2(M+N)+2\left\lfloor\frac{M}{P}\right\rfloor
          \left(M+N-P-P\left\lfloor\frac{M}{P}\right\rfloor\right) \\
      & \sim 2N(1+\lfloor M/P\rfloor) \qquad \text{as $  N \to \infty $}.
\end{align*}
Finally, substitution of this expression together with that for $  T_{1}  $
into (\ref{tag7}) leads us to the desired result.
\end{proof}

\begin{corollary}
For a tandem system with $  M  $ servers, Algorithm~3
using $  P = M+1  $ processors achieves linear speedup as the number of
customers $  N \to \infty $.
\end{corollary}

\begin{proof}
It follows from (\ref{tag6}) that with $  P=M+1  $ the speedup
$  S_{P} = O(P)  $ as $  N \to \infty $.
\end{proof}

\bibliographystyle{utphys}

\bibliography{Efficient_algorithms_for_tandem_queueing_system_simulation}

\end{document}